\documentclass[11pt,twoside]{article}
\usepackage[margin=1in]{geometry}
\usepackage{amsmath,amsthm,amssymb}
\usepackage[alphabetic]{amsrefs}
\usepackage[utf8]{inputenc}
\usepackage[T1]{fontenc}
\usepackage[colorlinks=true, linkcolor=blue, citecolor=blue]{hyperref}
\usepackage{microtype,mathtools,nicefrac,paralist}
\usepackage{tikz}

\usepackage{secdot}

\newtheorem{theorem}{Theorem}[section]
\newtheorem{lemma}[theorem]{Lemma}
\newtheorem{corollary}[theorem]{Corollary}

\newtheorem*{introthm}{Main Theorem}

\theoremstyle{definition}
\newtheorem{remark}[theorem]{Remark}
\newtheorem{definition}[theorem]{Definition}
\newtheorem{example}[theorem]{Example}

\newcommand\eg{{\em e.g., }}
\newcommand\ie{{\em i.e., }}

\newcommand{\dash}{\nobreakdash-\hspace{0pt}}

\DeclareMathOperator\ad{ad}

\DeclareMathOperator\Aut{Aut}

\DeclareMathOperator\ch{char}

\DeclareMathOperator\End{End}

\DeclareMathOperator\Gal{Gal}

\DeclareMathOperator\id{id}

\DeclareMathOperator\Mat{Mat}

\DeclareMathOperator\Rad{Rad}

\DeclareMathOperator\Span{span}

\DeclareMathOperator\Sym{Sym}

\DeclareMathOperator\Cent{\mathbf{Z}}
\newcommand\la{\langle}

\newcommand\ra{\rangle}

\newcommand\size[1]{\lvert #1\rvert}

\newcommand\FF{\mathbb{F}}

\newcommand\RR{\mathbb{R}}
\newcommand\ZZ{\mathbb{Z}}

\newcommand\op{\mathrm}
\renewcommand\cal{\mathcal}

\newcommand\al{\alpha}

\newcommand\lm{\lambda}

\newcommand\rt[1]{{\rm #1}}
\newcommand\affrt[1]{\tilde{\rm #1}}
\newcommand{\Fspace}{\Gamma}

\def\qed {{
	\parfillskip=0pt 
	\widowpenalty=10000 
	\displaywidowpenalty=10000 
	\finalhyphendemerits=0 
	\leavevmode 
	\unskip 
	\nobreak 
	\hfil 
	\penalty50 
	\hskip.2em 
	\null 
	\hfill 
	$\square$
	\par}} 

\allowdisplaybreaks
\defaultenum{\rm (i)}{\rm (a)}{i.}{A.}

\makeatletter
\renewcommand*\env@matrix[1][\arraystretch]{%
	\edef\arraystretch{#1}%
	\hskip -\arraycolsep
	\let\@ifnextchar\new@ifnextchar
	\array{*\c@MaxMatrixCols c}}
\makeatother

\begin{document}
\abovedisplayskip=0.5em
\belowdisplayskip=0.5em

\title{Jordan algebras and $3$-transposition groups}
\author{Tom De Medts \and Felix Rehren}
\date{October 5, 2015}
\maketitle

\begin{abstract}
	An idempotent in a Jordan algebra induces a Peirce decomposition of the algebra
	into subspaces whose pairwise multiplication satisfies certain fusion rules $\Phi(\nicefrac{1}{2})$.
	On the other hand, $3$\dash transposition groups $(G,D)$ can be algebraically characterised
	as Matsuo algebras $M_\al(G,D)$ with idempotents satisfying the fusion rules $\Phi(\al)$ for some $\al$.
	We classify the Jordan algebras $J$ which are isomorphic to a Matsuo algebra $M_{\nicefrac{1}{2}}(G,D)$,
	in which case $(G,D)$ is a subgroup of the (algebraic) automorphism group of $J$;
	the only possibilities are $G = \Sym(n)$ and $G = 3^2:2$.
	Along the way, we also obtain results about Jordan algebras associated to root systems.
\end{abstract}
\bigskip

	The celebrated theory of $3$-transposition groups,
	developed by B.~Fischer in the 1960s,
	captures the symmetric groups,
	certain finite groups of Lie type over fields of small characteristic,
	and the sporadic Fischer groups,
	and had a profound impact on 20th century group theory.
	The inherent combinatorial data has a formulation in terms of certain incidence geometries,
	called Fischer spaces,
	captured by F.~Buekenhout's well-known geometrical characterisation of $3$\dash transposition groups \cite{buek}; see also \cite{asch3}*{Section~18}.
	Recently an algebraic characterisation of $3$\dash transposition groups became available \cite{hrs}---%
	in terms of a special class of nonassociative algebras called {\em axial algebras}.
	(We refer to Definition~\ref{def:alg}\ref{alg:axial} below for a precise description.)

	These axial algebras have their origin
	in the Griess algebra, the 196884-dimensional nonassociative algebra
	whose automorphism group is the Monster,
	the largest sporadic simple group and itself a $6$-transposition group.
	Developments in vertex operator algebras
	provided a wealth of new examples of Griess-like algebras.
	Via foundational work of M.~Miyamoto and A.~A.~Ivanov \cites{miyamoto,ivanov},
	they were axiomatised as algebras generated by idempotents
	whose eigenvectors multiply according to some {\em fusion rules $\Phi$}.
	It is this phenomenon which was singled out in \cite{hrs} to study the much larger class
	of {\em axial algebras} (with respect to the fusion rules $\Phi$).
	Remarkably, the simplest case $\Phi(\al)$, in Table~\ref{tbl-jordan},
	of $\ZZ/2$-graded fusion rules
	exactly characterises $3$-transposition groups.

	\begin{table}[ht]
	\begin{center}
	\renewcommand{\arraystretch}{1.2}
	\setlength{\tabcolsep}{0.75em}
	\begin{tabular}{c|ccc}
			$\star$	& $1$ & $0$ & $\al$ \\
		\hline
			$1$ & $\{1\}$ & $\emptyset$ & $\{\al\}$ \\
			$0$ &	& $\{0\}$ & $\{\al\}$ \\
			$\al$ &	&	& $\{1,0\}$ \\
	\end{tabular}
	\end{center}
	\caption{Jordan fusion rules $\Phi(\al)$}
	\label{tbl-jordan}
	\end{table}

	We apply this new point of view to study a classical subject
	in nonassociative algebras: {\em Jordan algebras}.
	It is a well-known fact going back to the beginnings of the subject,
	that the eigenvalues of any idempotent in a Jordan algebra are $\{1,0,\nicefrac{1}{2}\}$
	and the eigenspaces satisfy the fusion rules $\Phi(\nicefrac{1}{2})$.
	This is the {\em Peirce decomposition}.
	It turns out that $\al = \nicefrac{1}{2}$
	leads to very distinguished behaviour for $\Phi(\al)$-axial algebras;
	for all other values of~$\al$,
	every $\Phi(\al)$-axial algebra
	is essentially the {\em Matsuo algebra} of a $3$-transposition group \cite{hrs}*{Theorem 5.4(b) and Theorem 6.3}.
	The wilder situation for $\Phi(\nicefrac{1}{2})$
	accommodates Matsuo algebras and Jordan algebras, but also admits further possibilities.

	In this paper, we answer the question:
	which Jordan algebras are Matsuo algebras?
    As a consequence, each of the resulting Jordan algebras is spanned by idempotents
    whose {\em Peirce reflections} (see Lemma~\ref{lem jordan axis}) generate a $3$-transposition group inside $\Aut(J)$.

	\begin{introthm}[%
		Theorems~\ref{thm jordan matsuo an},
		\ref{thm jordan matsuo pi3},
		\ref{thm jordan matsuo pi3-ch3},
		\ref{thm jordan matsuo}%
	]
		\label{thm intro}

		Let $\FF$ be a field, $\ch(\FF) \neq 2$,
		and let $J$ be a Jordan algebra over $\FF$ which is also a Matsuo algebra.
		Then $J$ is a direct sum of Matsuo algebras $J_i = M_{1/2}(G_i,D_i)$ corresponding to $3$-transposition groups $(G_i,D_i)$, where for each $i$,
		\begin{compactenum}
			\item
				either $G_i = \Sym(n)$,
				and $J_i$ is the Jordan algebra of $n \times n$ symmetric matrices over $\FF$ with zero row sums;
			\item
				or $G_i \cong 3^2:2$, and
				\begin{compactenum}
					\item
						either $\ch(\FF) \neq 3$ and $J_i$ is the Jordan algebra of hermitian $3\times3$ matrices
						over the quadratic \'etale extension $E=\FF[x]/(x^2+3)$,
					\item
						or $\ch(\FF) = 3$ and $J_i$ is a certain $9$-dimensional degenerate Jordan algebra with an $8$\dash dimensional radical.
				\end{compactenum}
		\end{compactenum}
	\end{introthm}

	The paper is organised as follows.
	Section~\ref{sec comb pre} recalls elementary facts
	on $3$-transposition groups, Fischer spaces and root systems.
	Section~\ref{sec alg pre} gives definitions and basic results
	for Jordan and Matsuo algebras and introduces fusion rules.
	Section~\ref{sec-sym} proves that the Matsuo algebra
	for $\Sym(n)$ is the Jordan algebra of zero-sum $n\times n$ symmetric matrices,
	and gives details of a construction of a Jordan algebra
	of projection matrices coming from a root system.
	Section~\ref{sec-P3} proves that the examples arising in (ii) of the Main Theorem are indeed Jordan algebras;
	in particular (a) involves recovering a Peirce decomposition in the Matsuo algebra,
	and for (b) we give a full description of a chain of ideals in the degenerate algebra.
	Section~\ref{sec-class} finally shows that these are the only Matsuo algebras
	which are also Jordan algebras.

	The second author would like to thank the Department of Mathematics of Ghent University,
	where this work was started, for their hospitality.
	Both authors are grateful to the referee for his valuable suggestions that greatly improved
	the exposition of the results.

\section{$3$-transposition groups, Fischer spaces and root systems}
\label{sec comb pre}
	In this section, we recall the definition of $3$-transposition groups and Fischer spaces,
	and we explain how each simply-laced root system gives rise to a Fischer space.
	\begin{definition}[\cite{asch3}]
		\label{def 3trposgp}
		A {\em $3$-transposition group} is a pair $(G,D)$ where $G$ is a group
		generated by $D\subseteq G$ a subset of involutions closed under conjugation
		such that, for any $c,d\in D$, $\size{cd}\leq3$.
	\end{definition}
	\begin{remark}
		Fischer's definition of a $3$-transposition group requires $D$ to be a single conjugacy class of involutions
		(see, \eg \cite{fischer}).
		This is not a severe restriction.
		Indeed, suppose that $(G,D)$ is a $3$-transposition group as in Definition~\ref{def 3trposgp} where $D = D_1 \cup D_2$
		is a $G$-invariant partition of $D$.
		Let $G_1 = \langle D_1 \rangle$ and $G_2 = \langle D_2 \rangle$; then $(G_1, D_1)$ and $(G_2, D_2)$ are again $3$-transposition groups,
		$[G_1,G_2]=1$, and $G = G_1 G_2$.
		Continuing this process, we can write $G$ as the central product of $3$-transposition groups $(G_i, D_i)$ where
		each $D_i$ is a single conjugacy class of involutions in $G_i$.
		See \cite{asch3}*{p.\@ 30, (8.2)}.
	\end{remark}

    To define Fischer spaces, we first need some general terminology and notation from incidence geometry.
	A {\em partial linear space}  is a pair $(\Fspace,\cal L)$,
	where $\Fspace$ is a set of {\em points}
	and $\cal L\subseteq 2^\Fspace$ is a set of {\em lines}
	such that any $\ell\in\cal L$ has size at least $2$,
	and such that any two distinct lines intersect in at most one point.
	Often $\Fspace$ alone refers to $(\Fspace,\cal L)$.
	A partial linear space $(\Fspace, \cal L)$ is called a {\em partial triple system}
	if every line has exactly $3$ points.
	In a partial triple system $\Fspace$,
	for any two collinear points $x,y\in\Fspace$
	there exists a unique line $\ell$ connecting $x$ and $y$,
	and a unique element denoted $x\wedge y\in\Fspace$
	such that $\ell = \{x,y,x\wedge y\}$.
	A subset $S$ of $\Fspace$ is called a {\em subspace} if it is closed under the operation $\wedge$,
	and if $T$ is any subset of $\Fspace$, then we write $\langle T \rangle$ for the
	{\em subspace generated by $T$}, \ie the smallest subspace of $\Fspace$ containing $T$.

	In such a $\Fspace$, for two distinct points $x,y\in\Fspace$
	we write $x\sim y$ if $x$ and $y$ are collinear,
	that is, if there exists a line containing $x$ and $y$,
	and $x\not\sim y$ otherwise.
	We can partition $\Fspace$ with respect to $x$ as
	$\{ x \} \cup x^\sim \cup x^{\not\sim}$,
	where $x^\sim=\{y\in\Fspace\mid x\sim y\}$
	and $x^{\not\sim}=\{y\in\Fspace\mid x\not\sim y\}$.
	A point $x \in \Fspace$ is called {\em isolated} if $x^\sim = \emptyset$.
    We call $\Fspace$ {\em nondegenerate} if it has no isolated points.

	We will need two specific examples of partial triple systems, namely the {\em dual affine plane of order $2$}, denoted by $\cal P_2^\vee$,
	and the {\em affine plane of order $3$}, denote by $\cal P_3$; see Figure~\ref{fig-projpl}.

	\begin{figure}[ht!]
	\begin{center}
	\begin{tikzpicture}
		\useasboundingbox (-3.0cm,3.0cm) rectangle (3.0cm,-3.5cm);
		\node at (0,-3.5cm) {$\cal P_2^\vee$};

		\tikzstyle{every node}=[draw,circle,fill=white,minimum size=4pt,inner sep=0pt]
		\draw
			node at +(120:2.1cm) [label=left:$1$] {}
			node at +(60:2.1cm) [label=right:$2$] {}
			node at +(0:2.1cm) [label=right:$3$] {}
			node at +(300:2.1cm) [label=right:$4$] {}
			node at +(240:2.1cm) [label=left:$5$] {}
			node at +(180:2.1cm) [label=left:$6$] {}
		;

		\draw[rounded corners=0.2cm] (120:2cm) -- (60:2cm) -- (180:2cm) -- cycle;
		\draw[rounded corners=0.2cm] (60:2cm) -- (0:2cm) -- (300:2cm) -- cycle;
		\draw[rounded corners=0.2cm] (180:2cm) -- (300:2cm) -- (240:2cm) -- cycle;
		\draw[rounded corners=0.2cm] (120:2cm) -- (0:2cm) -- (240:2cm) -- cycle;

	\end{tikzpicture}
	\quad\quad\quad
	\begin{tikzpicture}
		\useasboundingbox	(-1.0cm,1.0cm) rectangle (5.0cm,-5.5cm);
		\node at (2cm,-5.5cm) {$\cal P_3$};

		\tikzstyle{every node}=[draw=white,ultra thick,
			circle,fill=white,minimum size=6pt,inner sep=0pt]
		\draw (0,0) node (1) [label=left:$1$] {}
			++(0:2cm) node (2) [label=above:$2$] {}
			++(0:2cm) node (3) [label=right:$3$] {}
			++(270:2cm) node (6) [label=right:$6$] {}
			++(180:2cm) node (5) [label=above right:$5$] {}
			++(180:2cm) node (4) [label=left:$4$] {}
			++(270:2cm) node (7) [label=left:$7$] {}
			++(0:2cm) node (8) [label=below:$8$] {}
			++(0:2cm) node (9) [label=right:$9$] {}
		;
		\tikzstyle{every node}=[draw,
			circle,fill=white,minimum size=4pt,inner sep=0pt]
		\draw (1) node {}
			(2) node {}
			(3) node {}
			(4) node {}
			(5) node {}
			(6) node {}
			(7) node {}
			(8) node {}
			(9) node {}
		;

		\draw (1) -- (2) -- (3);
		\draw (4) -- (5) -- (6);
		\draw (7) -- (8) -- (9);
		\draw (1) -- (4) -- (7);
		\draw (2) -- (5) -- (8);
		\draw (3) -- (6) -- (9);
		\draw (1) -- (5) -- (9);
		\draw (3) -- (5) -- (7);

		\draw (1) to (6);
		\draw[shift=(1)] (6) .. controls (330:7.5cm) and (300:7.5cm)	.. (8);
		\draw (8) to (1);

		\draw (3) to (4);
		\draw[shift=(3)] (4) .. controls (210:7.5cm) and (240:7.5cm)	.. (8);
		\draw (8) to (3);

		\draw (7) to (2);
		\draw[shift=(7)] (2) .. controls (60:7.5cm) and (30:7.5cm)	.. (6);
		\draw (6) to (7);

		\draw (9) to (2);
		\draw[shift=(9)] (2) .. controls (120:7.5cm) and (150:7.5cm)	.. (4);
		\draw (4) to (9);
	\end{tikzpicture}
	\end{center}
	\caption{
		The dual affine plane $\cal P_2^\vee$
		and the affine plane $\cal P_3$
	}
	\label{fig-projpl}
	\end{figure}
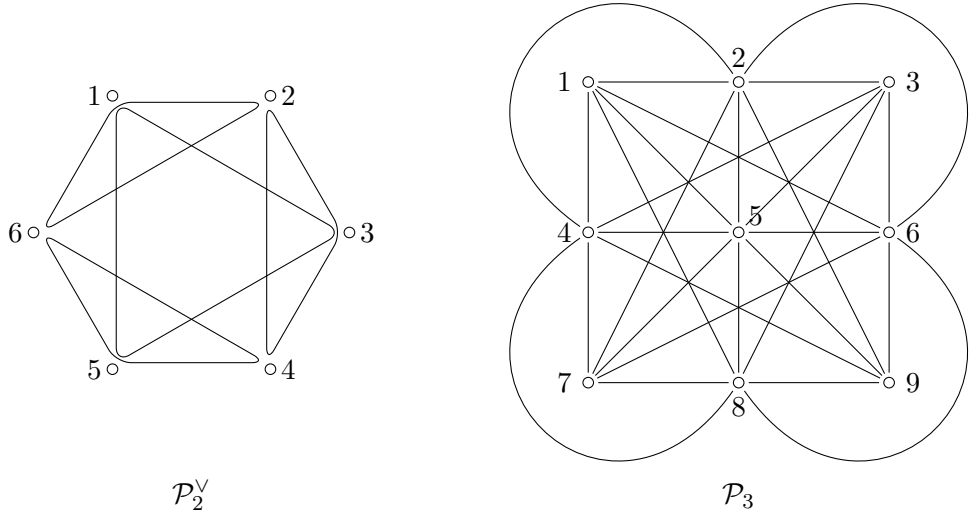

	\begin{definition}[\cite{asch3}*{Section 18}]
		\label{def-fischer-space}
		A {\em Fischer space} is a partial triple system
		for which, if $\ell_1,\ell_2$ are any two distinct intersecting lines,
		the subspace $\la\ell_1\cup\ell_2\ra$
		is isomorphic to the dual affine plane $\cal P_2^\vee$ of order $2$
		or to the affine plane $\cal P_3$ of order $3$.
		A Fischer space is said to be of {\em symplectic type}
		if $\cal P_3$ does not occur, \ie if the subspace $\la\ell_1\cup\ell_2\ra$
		is always isomorphic to~$\cal P_2^\vee$.
	\end{definition}

	There is a close connection between $3$-transposition groups and Fischer spaces:
	\begin{theorem}[\cite{buek}, \cite{asch3}*{Section~18}]\label{th:3tr-fisch}
		\begin{compactenum}
		    \item\label{3tr-fisch:i}
			Let $(G,D)$ be a $3$-transposition group.
			Let $\Fspace(G,D)$ be the partial triple system with point set $D$ and line set $\{ (c, d, c^d) \mid c,d \in D, \size{cd}=3 \}$.
			Then $\Fspace(G,D)$ is a Fischer space, and $G$ acts on $\Fspace(G,D)$ (by conjugation on the point set).
		    \item\label{3tr-fisch:ii}
			Let $\Fspace$ be a nondegenerate Fischer space.
			For each point $x \in \Fspace$, we define $\tau(x)$ to be the involutory permutation of $\Fspace$ fixing $x \cup x^{\not\sim}$ and
			interchanging the two points different from~$x$ on any line through $x$.
			Let $D = \{ \tau(x) \mid x \in \Fspace \}$.
            Then $D \subseteq \Aut(\Fspace)$, and $(\langle D \rangle, D)$ is a $3$-transposition group with trivial center.
            Moreover, $\Fspace(\langle D \rangle, D) \cong \Fspace$.
		\end{compactenum}
	\end{theorem}
    \begin{proof}
        Statement (i) is \cite{asch3}*{(18.2)}; statement (ii) is \cite{asch3}*{(18.1)}.
        See also \cite{asch3}*{p.\@~92, Example (4)}, and in particular the observation that the kernel of the action of $G$
        on its Fischer space $\Fspace(G,D)$ induced by conjugation on $D$ is precisely the center $\Cent(G)$ of $G$.
    \end{proof}

	\begin{example}\label{ex:fisch}
		\begin{compactenum}
		    \item\label{fisch:S4}
			Let $G = \Sym(4)$ and let $D$ be the conjugacy class of all transpositions in $G$.
			Then $(G,D)$ is a $3$-transposition group with Fischer space $\cal P_2^\vee$.
            See \cite{asch3}*{p.\@~93, Example (5)}.
		    \item\label{fisch:P3}
			Let $G = 3^2:2$, an elementary abelian group of order $9$ extended by an inverting involution,
            and let $D$ be the conjugacy class of this involution (which has size $9$).
			Then $(G,D)$ is a $3$-transposition group with Fischer space $\cal P_3$.
            See \cite{asch3}*{p.\@~93, Example (6)}.
		\end{compactenum}
	\end{example}

	We will also make use of {\em root systems},
	for which we refer the reader to any book on the subject,
	such as \cite{hum}*{Chapter III}.
	We will write $\rt R$ for a root system,
	and we write $\rt R = \rt R_+\cup \rt R_-$ for some partition of $\rt R$
	such that $\rt R_- = -\rt R_+$.
	(Typically, $\rt R_+$ would be the set of positive roots of $R$, but we do not require this.)

	\begin{lemma}
		\label{lem fisch roots}
		Suppose $\rt R$ is a simply-laced root system
		and $\Fspace$ is the partial triple system with point set $\rt R_+$
		and lines $\{r,s,t\}$ for distinct roots $r,s,t$
		spanning a root system of type $\rt A_2$.
		Then $\Fspace$ is a Fischer space of symplectic type,
		which we will denote by $\Fspace(\rt R)$.
	\end{lemma}
	\begin{proof}
		Suppose that $\rt R$ is a simply-laced root system inside a Euclidean vector space $V$.
		We first observe that $\Fspace$ as defined is a partial triple system.
		Indeed, if two lines $\ell_1,\ell_2$ intersect in two points $r,s$,
		then $r,s$ span a root system of type $\rt A_2$
		in a subspace $U\subseteq V$ of dimension $2$.
		Then $U\cap\rt R_+$ has size $3$,
		so the third point in both $\ell_1$ and $\ell_2$ is uniquely determined,
		so that $\ell_1=\ell_2$.
		Thus $\Fspace$ is a partial triple system.

		Now suppose that $\ell_1,\ell_2$ are two distinct intersecting lines,
		say $\ell_1 = \{r,s,t\}$ and $\ell_2 = \{r,u,v\}$.
		Therefore the vector space $U = \Span(\ell_1\cup\ell_2) \subseteq V$ is $3$-dimensional
		and, as $\rt A_3$ is the only irreducible simply-laced root system of rank $3$,
		$\ell_1\cup\ell_2$ span a root system of type~$\rt A_3$,
		which has $6$ positive roots and $4$ distinct subspaces $\rt A_2$.
		It is thus easy to check that $\la\ell_1\cup\ell_2\ra$ in $\Fspace$
		is isomorphic to $\cal P_2^\vee$ with $6$ points and $4$ lines.
	\end{proof}


	\begin{example}\label{ex:An}
			Let $G = \Sym(n)$ (with $n \geq 2$) and let $D$ be the conjugacy class of all transpositions in $G$.
			Then $(G,D)$ is a $3$-transposition group with Fischer space $\Fspace(\rt A_{n-1})$.
            Notice that this generalizes Example~\ref{ex:fisch}\ref{fisch:S4}.
	\end{example}

\section{Jordan algebras, Matsuo algebras and fusion rules}
\label{sec alg pre}
	From now on, we will always assume that $\FF$ is a commutative field
	of characteristic $\ch(\FF)$ not $2$.
	An {\em $\FF$-algebra} is a vector space over $\FF$ equipped with an $\FF$-bilinear multiplication.
	We do not require our algebras to be associative or unital,
	but all algebras in this text are commutative.

	We will be interested in two kinds of algebras: Jordan algebras and Matsuo algebras.
	\begin{definition}
		A {\em Jordan algebra} over $\FF$ is a commutative algebra
		$J$ over $\FF$ such that $(ab)(aa) = a(b(aa))$ for all $a,b \in J$.
	\end{definition}
	The theory of Jordan algebras is present in many different areas of mathematics, and we refer the reader
	to McCrimmon's book \cite{mccrimmon} for an excellent introduction to the subject.
	Associative algebras are an important source of Jordan algebras;
	we present some specific examples that we will need in the sequel.
	\begin{example}\label{ex jordan}
		\begin{compactenum}
			\item
				If $A$ is an associative $\FF$-algebra,
				then $A^+$ with the same underlying vector space
				and {\em Jordan product} $x\bullet y = \frac{1}{2}(xy+yx)$
				is a Jordan algebra over $\FF$ \cite{aaa}.
				If $A$ is unital, then $A^+$ is also unital, with the same unit.
			\item\label{ex jordan herm}
				If $A$ is an associative $\FF$-algebra with involution $*$, then
				the subspace $\mathcal{H}(A, *) = \{ x \in A \mid x^* = x \}$
				forms a Jordan subalgebra of $A^+$;
				it is called the Jordan algebra of {\em hermitian elements} of $A$ (with respect to $*$).
                If $*$ is $\FF$-linear, then $\mathcal{H}(A, *)$ is again a Jordan algebra over $\FF$.
                (In general, it is only a Jordan algebra over the subfield $\operatorname{Fix}_\FF(*)$.)
			\item\label{ex jordan zero-sum}
				Let $A$ be the associative $\FF$-algebra of $n\times n$ matrices whose rows and columns all sum to $0$,
				and let $t$ be the usual matrix transposition.
				Then by~\ref{ex jordan herm}, $\cal H(A,t)$ is a Jordan algebra, which we will refer to as
				the {\em Jordan algebra of symmetric zero-sum $n\times n$ matrices over $\FF$}.
				This algebra is unital if and only if $n \neq 0$ in $\FF$, in which case the identity element is
				the matrix with each diagonal entry equal to $(n-1)/n$ and each non-diagonal entry
				equal to $-1/n$.
		\end{compactenum}
	\end{example}

	Matsuo algebras are a much more recent object of study.
	They occured first in \cite{matsuo} in the context of $3$-transposition groups related to vertex operator algebras.

	\begin{definition}[\cite{matsuo}]
        \begin{compactenum}
            \item
                Let $\al\in\FF$ and $\Fspace$ a partial triple system.
                The {\em Matsuo algebra} $M_\al(\Fspace)$
                is the $\FF$-algebra with basis $\Fspace$,
                where the multiplication of two basis elements $x,y\in\Fspace$ is given by
                \begin{equation}
                    \label{eq matsuo mult}
                    xy = \begin{cases}
                        x & \text{ if } x = y \\
                        0 & \text{ if } x\not\sim y \\
                        \frac{\al}{2}(x+y-x\wedge y) & \text{ if } x\sim y.
                    \end{cases}
                \end{equation}
                We will view $\Fspace$ as embedded in $M_\al(\Fspace)$.
                Hence any $x\in\Fspace$ is an idempotent, that is, $xx=x$.
            \item
                Let $\al\in\FF$ and let $(G,D)$ be a $3$-transposition group.
                The {\em Matsuo algebra} $M_\al(G,D)$ is defined to be the Matsuo algebra $M_\al(\Fspace(G,D))$ of the Fischer space
                $\Fspace(G,D)$ as defined in Theorem~\ref{th:3tr-fisch}\ref{3tr-fisch:i}.
                In particular, $\dim_\FF M_\al(G,D) = \size{D}$.
        \end{compactenum}
	\end{definition}

	To avoid degeneracy, from now on we assume $\al\neq1,0$.
	\begin{lemma}
		\label{lem-eigv-matsuo}
		The eigenspaces of $x\in\Fspace$ in $M_\al(\Fspace)$ are
		\begin{align}
			& \la x\ra, \text{ its }1\text{-eigenspace,} \\
			& \la y+x\wedge y - \al x\mid y\sim x\ra\oplus\la y \mid y\not\sim x\ra, \text{ its }0\text{-eigenspace, and} \\
			& \la y-x\wedge y\mid y\sim x\ra, \text{ its }\al\text{-eigenspace.}
		\end{align}
		The algebra $M_\al(\Fspace)$ decomposes as a direct sum of these eigenspaces for any $x\in\Fspace$.
	\end{lemma}
	\begin{proof}
		The points of $\Fspace$ form a basis for $A = M_\al(\Fspace)$.
		Evidently $x$ is a $1$-eigenvector,
		and $x^{\not\sim}$ is a set of $0$-eigenvectors.
		Now partition $x^\sim$ into sets $\{y,x\wedge y\}$ for $y\sim x$.
		Then the subspace $\la x,y,x\wedge y\ra$ of $A$
		is spanned by $x,y-x\wedge y$ and $y+x\wedge y - \al x$,
		and these are $1,\al,0$-eigenvectors of $x$ respectively:
		$xx = x$, and
		\begin{align}
			x(y-x\wedge y) & = \frac{\al}{2}(x+y-x\wedge y-x-x\wedge y+y) = \al(y-x\wedge y), \\
			x(y+x\wedge y - \al x) & = \frac{\al}{2}(x+y-x\wedge y +x+x\wedge y-y) - \al x
			= \al x - \al x = 0.
		\end{align}
		Thus every pair $\{y,x\wedge y\}$ gives a pair of $0,\al$-eigenvectors of $x$,
		and so we have a bijection between a basis of $A$
		and a collection of linearly independent $1,0,\al$-eigenvectors of $x$.
	\end{proof}

    We now turn to the specific case of Fischer spaces.
	\begin{lemma}[\cite{hrs}*{Theorem 5.3, Proposition 5.4}]
		\label{lem inj miyamoto}
		Let $\Fspace$ be a nondegenerate Fischer space, and let $M_\alpha(\Fspace)$ be the corresponding $\FF$-Matsuo algebra.
		For each $x \in \Fspace$, let $\tau(x)$ be the $\FF$-linear automorphism of $M_\al(\Fspace)$
		acting on the eigenspaces of $x$ by
		\begin{equation}\label{eq:miyamoto}
			y^{\tau(x)} = \begin{cases}
				y & \text{ if } xy = 0 \text{ or } xy = y, \\
				-y & \text{ if } xy = \al y.
			\end{cases}
		\end{equation}
		Then each $\tau(x)$ is an involution,
		$\size{\tau(x)\tau(y)}\leq3$ for any $x,y\in\Fspace$,
		and the map
		\[ \Fspace\to\Aut(M_\al(\Fspace)),\quad x\mapsto\tau(x) \quad \]
		is an injection.
		\qed
	\end{lemma}
	The automorphism $\tau(x)$ is known as a {\em Miyamoto involution};
	restricted to the points $\Fspace\subseteq M_\al(\Fspace)$,
	these $\tau(x)$ are the same as the involutions occuring in the statement of Theorem~\ref{th:3tr-fisch}\ref{3tr-fisch:ii}.
	In particular, if we set $D_\Fspace = \{\tau(x)\mid x\in\Fspace\}$ and $G_\Fspace = \la D_\Fspace\ra$,
	then $(G_\Fspace,D_\Fspace)$ is a $3$-transposition group, and
    Lemma~\ref{lem inj miyamoto} realizes the group $G_{\Fspace}$ as a subgroup of the automorphism group of the algebra $M_\al(\Fspace)$.

    Conversely, if we start from a $3$-transposition group $(G,D)$, then
    the elements of $G$ act $\FF$-linearly on the corresponding Matsuo algebra $A = M_\alpha(G,D)$ by conjugating the elements of the basis $\Fspace = D$.
    Recall from Theorem~\ref{th:3tr-fisch} that this action is not faithful in general, and its kernel is precisely the center $\Cent(G)$.
    We thus get a faithful linear representation of $G/\Cent(G)$ as a subgroup of the automorphism group of~$A$.

	\medskip
	We will now define what it means for a decomposition of an algebra (w.r.t.\@ an idempotent) to satisfy certain fusion rules.
	\begin{definition}\label{def:alg}
        Let $A$ be an $\FF$-algebra.
		\begin{compactenum}
		    \item
			{\em Fusion rules} are a pair $(\Phi,\star)$,
			consisting of a set $\Phi\subseteq \FF$ of {\em eigenvalues} lying in the field $\FF$
			and a mapping $\star\colon\Phi\times\Phi\to2^\Phi$.
			We also use $\Phi$ to refer to $(\Phi,\star)$.
		    \item
			For $x\in A$,
			we call the eigenvalues, eigenvectors and eigenspaces
			of the adjoint map $\ad(x)\in\End(A)$
			the eigenvalues, eigenvectors and eigenspaces of $x$, respectively.
			The $\al$-eigenspace of $x$ in $A$ is denoted $A^x_\al = \{a\in A\mid xa = \al a\}$.
			By extension, if $\Psi\subseteq \FF$ is a set,
			we write $A^x_\Psi = \bigoplus_{\phi\in\Psi}A^x_\phi$,
			and $A^x_\emptyset = \{ 0 \}$.
		    \item
			An idempotent $e$ in an algebra $A$ is a {\em $\Phi$-axis}
			if $\ad(e)$ is diagonalisable, takes all its eigenvalues in $\Phi$,
			so that $A$ decomposes into the direct sum of $\ad(e)$-eigenspaces
			\begin{equation}
				A = \bigoplus_{\phi\in\Phi} A^e_\phi,
			\end{equation}
			and the multiplication of eigenvectors {\em satisfies the fusion rules $\Phi$}, \ie
			\begin{equation}
				A^e_\phi A^e_\psi\subseteq A^e_{\phi\star\psi}.
			\end{equation}
			In other words,
			$xy\in A^e_{\phi\star\psi} = \bigoplus_{\chi\in\phi\star\psi} A^e_\chi$
			for all $x\in A^e_\phi$ and $y\in A^e_\psi$.
		    \item\label{alg:axial}
			The algebra $A$ is called {\em $\Phi$-axial} (or simply {\em axial}) if it is generated (as an algebra) by $\Phi$-axes.
		\end{compactenum}
	\end{definition}

    \begin{example}
            Let $\al \in \FF \setminus \{ 0,1 \}$, and let $\Phi = \{ 1,0,\al \}$.
            Let $\star$ be the symmetric map $\star \colon \Phi \times \Phi \to 2^\Phi$ as given by Table~\ref{tbl-jordan}.
            These fusion rules $(\Phi, \star)$ are denoted by $\Phi(\al)$ and are called {\em Jordan fusion rules}
            because the eigenspaces of an idempotent in a Jordan algebra
            multiply according to these fusion rules for $\al=\nicefrac{1}{2}$; see Lemma~\ref{lem jordan axis} below.
    \end{example}

    It turns out that both Matsuo algebras and Jordan algebras satisfy fusion rules $\Phi(\al)$.
    \begin{theorem}[\cite{hrs}*{Theorem 6.4}]
        Let $A = M_\alpha(\Fspace)$ be a Matsuo algebra.
        Then every element of the basis $\Fspace$ is a $\Phi(\al)$-axis.
        In particular, $A$ is a $\Phi(\al)$-axial algebra.
        \qed
    \end{theorem}

	\begin{lemma}[\cite{jac}*{Chapter III, Lemma 1.1, p.\@ 119}]
		\label{lem jordan axis}
		Let $J$ be a Jordan algebra.
		Then every idempotent $e \in J$ is a $\Phi(\nicefrac{1}{2})$-axis;
        this is known as the {\em Peirce decomposition} of $J$ with respect to $e$.
        The corresponding involutions $\tau(e) \in \Aut(J)$ as defined in~\eqref{eq:miyamoto} are known as {\em Peirce reflections};
        see~\cite{mccrimmon}*{Exercise 8.1.3, p.\@ 238}.
		\qed
	\end{lemma}

	It is clear that not every Jordan algebra is axial (for instance, a Jordan division algebra has no non-trivial idempotents).
    The following results, together with Lemma~\ref{lem jordan axis},
	give a largely satisfactory answer
	to the question of which Jordan algebras are $\Phi(\nicefrac{1}{2})$-axial.
	Recall that $x\in A$ is {\em nilpotent}
	if there exists some integer $n$ such that $x^n = 0$,
	and that an ideal $I$ is {\em solvable}
	if there is an integer $k \geq 0$ such that $I^{2^k} = 0$,
	where $I^{2^k}$ is defined inductively by $I^{1} = I$, $I^{2} = \{ \sum_i a_i b_i \mid a_i,b_i \in I \}$ and $I^{2^k} = (I^{2^{k-1}})^2$.
    (See also \cite{jac}*{Chapter V, Section 2}.)
	\begin{theorem}
		\label{thm jordan idempots}
		Suppose that $J$ is a finite-dimensional Jordan algebra over $\FF$.
		\begin{compactenum}
			\item {\em (\cite{aaa}*{Lemma 4})}
				If $a\in J$ is not nilpotent,
				then $\FF[a]\subseteq J$ contains a nonzero idempotent.
			\item {\em (\cite{aaa}*{Theorem 5})}
				There is a unique largest solvable ideal $\Rad(J)$ of $J$,
				called its {\em radical}.
				All elements of $\Rad(J)$ are nilpotent,
				and $J/\Rad(J)$ is a semisimple Jordan algebra.
			\item {\em (\cite{jac}*{Chapter VIII, Section 3, Lemma 2})}
				If $J$ is semisimple and $\FF$ is algebraically closed,
				then $J$ is spanned by idempotents.
				\qed
		\end{compactenum}
	\end{theorem}

\section{Construction of Jordan algebras from root systems}
	\label{sec-sym}

    Our main goal in this section is to show that the Matsuo algebras arising from symmetric groups are indeed Jordan algebras;
    see Theorem~\ref{thm jordan matsuo an} below.
    Recall from Example~\ref{ex:An} that for each $n \geq 2$, the $3$-transposition group $(G,D)$ with $G = \Sym(n)$ and $D$ the conjugacy class of all
    transpositions, has corresponding Fischer space $\Fspace(G,D) = \Fspace(\rt A_{n-1})$.
	Our main tool will be the construction of Jordan algebras arising from root systems.

    \medskip

	So suppose that $\rt R$ is an arbitrary root system of rank $n$;
	recall that this implies that $\rt R$ spans $\RR^n$
	with Euclidean form $(\cdot,\cdot)$.
	Based on the description of the root systems with coefficients from $\tfrac{1}{2}\ZZ$ only (see, \eg \cite{carter}*{Appendix}),
	we can consider the root systems inside $V = \FF^{n+1}$, equipped with the standard inner product $(\cdot,\cdot)$;
	recall that we assume $\ch(\FF) \neq 2$.
	We also have to exclude $\ch(\FF)=3$ when $\rt R$ contains $\rt G_2$ as one of its irreducible components.

	For $v\in V$ with $vv^t \neq 0$, write $m_v$ for the projection matrix
	of the $1$-dimensional subspace $\la v\ra\subseteq V$,
	\ie $m_v = \frac{1}{vv^t} v^t v$.
	Let $J(\rt R)$ be the subspace spanned by the projection matrices $\{ m_a\mid a\in R \}$.
	Notice that $J(\rt R)$ is {\em not} closed under matrix multiplication in general.
	Also note that $m_a = m_{-a}$,
	so it suffices to consider the projection matrices
	for a set $\rt R_+$ of positive roots.

	\begin{example}\label{ex proj An}
		Let $V = \FF^{n}$ with standard ordered basis $v_1,\dotsc,v_{n}$.
		There is an embedding
		\begin{equation}
			\label{eq An in Rn+1}
			\rt (\rt A_{n-1})_+ = \{a_{ij} = v_i-v_j\mid 1\leq j<i\leq n \}\subseteq V,
		\end{equation}
		and its projection matrices are
		\begin{equation}
			\label{eq proj mat an}
			m_{a_{ij}} = \frac{1}{2}(e_{ii} - e_{ij} - e_{ji} + e_{jj}),
		\end{equation}
		where $e_{ij}$ is the $n \times n$ matrix
		whose entries are $0$ everywhere except in position $(i,j)$ where it has entry $1$.
	\end{example}

	\begin{lemma}
		\label{lem rootsys proj mult}
		Let $\rt R$ be an irreducible root system.
		Suppose that $a,b\in \rt R$ are two roots with $\lVert b \rVert \geq \lVert a \rVert$,
		choose $k\in\{1,2,3\}$ to satisfy $\lVert b \rVert = \sqrt{k} \, \lVert a \rVert$,
		and assume that $\ch(\FF) \neq 3$ when $k=3$ (i.e., when $\rt R = \rt G_2$).
		Let $\langle a,b \rangle$ denote the root system generated by $a$ and $b$.
		Then
		\begin{equation}
			m_a\bullet m_b = \begin{cases}
				m_a & \text{if } a = \pm b,
				\text{ \ie }\la a,b\ra\cong\rt A_1, \\
				0 & \text{if } (a,b)=0;\;
				\la a,b\ra\cong\rt A_1\times\rt A_1, \\
				\frac{1}{4}(m_a + k \cdot m_b - m_c)
				& \text{otherwise: } c\in \la a,b \ra \cap\{a\pm b\};\;
				\la a,b\ra\cong\rt A_2, B_2 \text{ or } G_2.
			\end{cases}
		\end{equation}
		In particular, $J(\rt R) = \langle m_a \mid a \in \rt R \rangle$ with the Jordan product $\bullet$ is a Jordan algebra.
	\end{lemma}
	\begin{proof}
		Projections are idempotents,
		so that $m_a\bullet m_{-a} = m_a\bullet m_a = m_a^2 = m_a$ for all $a\in \rt R$.

		Suppose that $(a,b) = 0$,
		so that $a$ and $b$ are orthogonal with respect to the inner product.
		Then $m_a$ and $m_b$ are mutually orthogonal projections in $V$,
		and hence $m_a m_b = m_b m_a = 0$, so $m_a\bullet m_b = 0$.

		Suppose now (by replacing $b$ with its negative if necessary)
		that the angle between $a$ and $b$ in $\rt R$ lies strictly between $\pi/2$ and $\pi$.
		Then an inspection of the possible root systems of rank $2$ shows that $a$ and $b$ are the fundamental roots of
		a root system of type $A_2$, $B_2$ or $G_2$, so in particular $a+b \in \rt R$.
		We now deal with these three possible cases separately.

		Assume first that $a$ and $b$ generate a root system of type $\rt A_2$.
		Then without loss of generality, we may assume that $\rt R = \rt A_2$
		in $\FF^3$ according to Example~\ref{ex proj An}, and
		\begin{gather}
			a = (1, -1, 0), \quad b = (0, 1, -1), \quad a+b = (1, 0, -1), \\
			m_a = \tfrac{1}{2}\begin{psmallmatrix*}[r]
					1 & -1 & 0 \\
					-1 & 1 & 0 \\
					0 & 0 & 0 \\
				\end{psmallmatrix*}, \quad
			m_b = \tfrac{1}{2}\begin{psmallmatrix*}[r]
					0 &	0 & 0 \\
					0 & 1 & -1 \\
					0 & -1 & 1 \\
				\end{psmallmatrix*}, \quad
			m_{a+b} = \tfrac{1}{2}\begin{psmallmatrix*}[r]
					1 & 0 & -1 \\
					0 & 0 & 0 \\
					-1 & 0 & 1 \\
				\end{psmallmatrix*}.
		\end{gather}
		Indeed $m_a\bullet m_b = \frac{1}{4}(m_a + m_b - m_c)$
		in this representation,
		and hence in general.

		Assume next that $a$ and $b$ generate a root system of type $\rt B_2$.
		Then without loss of generality, we may assume that $\rt B_2$
		in $\FF^2$, and
		\begin{gather}
			a = (1, 0), \quad b = (-1, 1), \quad a+b = (0, 1), \\
			m_a = \begin{psmallmatrix*}[r]
					1 & 0 \\
					0 & 0 \\
				\end{psmallmatrix*}, \quad
			m_b = \tfrac{1}{2}\begin{psmallmatrix*}[r]
					1 & -1 \\
					-1 & 1 \\
				\end{psmallmatrix*}, \quad
			m_{a+b} = \begin{psmallmatrix*}[r]
					0 & 0 \\
					0 & 1 \\
				\end{psmallmatrix*}.
		\end{gather}
		Here we see that $m_a\bullet m_b = \frac{1}{4}(m_a + 2m_b - m_c)$.

		Assume finally that $a$ and $b$ generate a root system of type $\rt G_2$, and that $\ch(\FF) \neq 3$.
		In the standard construction of $\rt G_2$ in $\FF^3$,
		\begin{gather}
			a = (1, -1, 0), \quad b = (-1, 2, -1), \quad a+b = (0, 1, -1), \\
			m_a = \tfrac{1}{2}\begin{psmallmatrix*}[r]
					1 & -1 & 0 \\
					-1 & 1 & 0 \\
					0 & 0 & 0 \\
				\end{psmallmatrix*}, \quad
			m_b = \tfrac{1}{6}\begin{psmallmatrix*}[r]
					1 & -2 & 1 \\
					-2 & 4 & -2 \\
					1 & -2 & 1 \\
				\end{psmallmatrix*}, \quad
			m_{a+b} = \tfrac{1}{2}\begin{psmallmatrix*}[r]
					0 &	0 & 0 \\
					0 & 1 & -1 \\
					0 & -1 & 1 \\
				\end{psmallmatrix*}.
		\end{gather}
		Similarly, $m_a\bullet m_b = \frac{1}{4}(m_a + 3m_b - m_c)$.
	\end{proof}

	The elements $m_a \in J(\rt R)$, for $a\in\rt R$, are linearly independent when $\rt R = \rt A_n$ (see Example~\ref{ex proj An}),
	but this is not true in general.
	More specifically, we have the following result.

	\begin{lemma}
		\label{lem dim rootsys projalg}
		Let $\rt R$ be an irreducible root system of rank $n$,
		and assume $\ch(\FF) \neq 3$ if $\rt R = \rt G_2$.
		Then $J(\rt R)$ has dimension $\frac{1}{2}n(n+1)$.
	\end{lemma}
	\begin{proof}
		Let $\Pi = \{ \al_1,\dots,\al_n \}$ be a set of fundamental roots for $\rt R$, and view the roots as elements of the row space $\FF^{n+1}$.
		We claim that the set
		\[ B = \{ \al_i^t \al_j + \al_j^t \al_i \mid 1 \leq i \leq j \leq n \} \]
		consisting of $\frac{1}{2}n(n+1)$ matrices forms a basis for $J(\rt R)$.

		As any projection matrix $m_a$ is a scalar multiple of $a^ta$,
		$J(\rt R)$ is spanned by the set $C = \{ a^t a \mid a \in \rt R_+ \}$.
		Since any $a \in \rt R_+$ is an integral linear combination $\sum_{i=1}^n\lm_i\al_i$ of the fundamental roots,
		$J(\rt R)$ is contained in $\Span(B)$:
		as $\ch(\FF)\neq 2$ we have $\al_i^t\al_i\in\Span(B)$, and
		\begin{equation}
			a^ta = \sum_{i=1}^n\lm_i^2\al_i^t\al_i + \sum_{\mathclap{1\leq i<j\leq n}}\lm_i\lm_j(\al_i^t\al_j + \al_j^t\al_i).
		\end{equation}

		To prove the converse, that $B$ is contained in $\Span(C)$,
		we use induction on the distance $d = d(i,j)$ between the nodes $i$ and $j$,
		for an element $\al_i^t\al_j + \al_j^t\al_i\in B$,
		in the Dynkin diagram of $\rt R$ formed by $\Pi$.
		The distance is well-defined because $\rt R$ is irreducible.
		The claim is obvious for $d=0$, that is, $i=j$.
		Assume now that each element $\al_i^t \al_j + \al_j^t \al_i$ where $d(i,j) < d$ is contained in $\Span(C)$,
		and consider an element	$\al_i^t \al_j + \al_j^t \al_i$ with $d(i,j) = d$.
		Let $a \in \rt R_+$ be an arbitrary positive root that is an integral linear combination of the roots on the unique path from $i$ to $j$
		in the Dynkin diagram, having a non-zero coefficient for both $i$ and $j$.
		Write $a = \lm_i \al_i + \dots + \lm_j \al_j$, so $\lm_i \lm_j \neq 0$.
		It now suffices to expand the expression for $a^t a \in C$ to see that $a^t a$ is the sum of $\lm_i \lm_j (\al_i^t \al_j + \al_j^t \al_i)$
		and terms that are in $\Span(C)$ by the induction hypothesis, and we conclude that $(\al_i^t \al_j + \al_j^t \al_i) \in \Span(C)$ as well.

		This shows that $\Span(B) = \Span(C) = J(\rt R)$, and it remains to show that the elements of $B$ are linearly independent.
		This is clear, however, because the set $\Pi$ extends to a basis of the vector space $\FF^{n+1}$, and with respect to this basis,
		distinct elements of $B$ have nonzero entries in distinct positions.
	\end{proof}

    We are now ready to show that the Matsuo algebras corresponding to the symmetric groups are indeed Jordan algebras.
	\begin{theorem}
		\label{thm jordan matsuo an}
		Let $n \geq 1$ be an integer, let $G = \Sym(n)$ with $D = (12)^G$, and let $A$ be the Matsuo algebra
        $M_{\nicefrac{1}{2}}(G,D) = M_{\nicefrac{1}{2}}(\Fspace(\rt A_{n-1}))$ over $\FF$.
        Then $A$ is isomorphic to the Jordan algebra of symmetric zero-sum $n\times n$ matrices over $\FF$
		from Example~\ref{ex jordan}\ref{ex jordan zero-sum}.
	\end{theorem}

	\begin{proof}
		Let $\rt A_{n-1}$ be embedded in $V\cong\FF^{n}$ as in \eqref{eq An in Rn+1}.
		By the previous Lemmas~\ref{lem rootsys proj mult} and~\ref{lem dim rootsys projalg},
		$J(\rt A_{n-1})$ is a Jordan algebra of dimension $\frac{1}{2}n(n-1)$
		which, since it satisfies the same multiplication,
		is a quotient of the Matsuo algebra $M_{\nicefrac{1}{2}}(\Fspace(\rt A_{n-1}))$.
		However,
        \[ \dim M_{\nicefrac{1}{2}}(\Fspace(\rt A_{n-1})) = \size{\rt (A_{n-1})_+} = \tfrac{1}{2}n(n-1) = \dim J(\rt A_{n-1}) , \]
		and therefore $J(\rt A_{n-1})$ is isomorphic to $M_{\nicefrac{1}{2}}(\Fspace(\rt A_{n-1}))$.
		We see from \eqref{eq proj mat an} that the elements of $J(\rt A_{n-1})$
		are symmetric zero-sum matrices, and hence $J(\rt A_{n-1})$ is a subalgebra of the Jordan algebra
		of all symmetric zero-sum $n \times n$ matrices.
		As the latter algebra also has dimension $\frac{1}{2}n(n-1)$, the result follows.
	\end{proof}

    \begin{remark}
        \begin{compactenum}[\rm (i)]
            \item
                For the remaining irreducible simply-laced root systems, \ie those of type $\rt D_n$ and $\rt E_n$,
                the resulting Matsuo algebra is not a Jordan algebra, but the same argument as above now shows that it has a large quotient
                (of dimension $\tfrac{1}{2}n(n+1)$) which is a Jordan algebra,
                namely the Jordan algebra of all symmetric $n \times n$ matrices over $\FF$.
            \item
                When $\rt A_n$ can be embedded in $\FF^n$ (which is an assumption depending on $n$ and on the field~$\FF$),
                the algebra of symmetric zero-sum $(n+1) \times (n+1)$ matrices over $\FF$ is also isomorphic to
                the algebra of all symmetric $n \times n$ matrices over $\FF$.
        \end{compactenum}
    \end{remark}

\section{The Matsuo algebra of the affine plane $\cal P_3$}
	\label{sec-P3}

	For the Matsuo algebra of $\cal P_3$, corresponding to the $3$-transposition group $(G,D)$ with $G = 3^2 : 2$ as in Example~\ref{ex:fisch}\ref{fisch:P3},
	it will turn out that the situation completely degenerates when the characteristic
	of the underlying field $\FF$ is equal to $3$.
	Therefore, we distinguish the case $\ch(\FF) \neq 3$ from the case $\ch(\FF) = 3$.
    Our main results are given by Theorems~\ref{thm jordan matsuo pi3} and~\ref{thm jordan matsuo pi3-ch3} below.

\subsection{The case $\ch(\FF) \neq 3$}\label{ss:not3}

	Recall the affine plane $\cal P_3$ from Figure~\ref{fig-projpl} on page~\pageref{fig-projpl},
	and let $A$ be the Matsuo algebra $A = M_{\nicefrac{1}{2}}(\cal P_3)$
    over a field $\FF$ with $\ch(\FF) \neq 2,3$.
	For each $i \in \{ 1,\dots, 9 \}$,
	we let $p_i \in A$ be the generator
	corresponding to the point $i$ in Figure~\ref{fig-projpl}.

	\begin{lemma}
		\label{lem pi3 unital}
		The algebra $A$ is unital, with $\id = \frac{1}{3} \sum_{i=1}^{9} p_i$.
	\end{lemma}
	\begin{proof}
		Let $z = \sum_{i=1}^{9} p_i$.
		By symmetry and linearity, it suffices to verify that $z p_1 = 3 p_1$.
		Indeed,
		\begin{align*}
			z p_1
			= p_1 + \tfrac{1}{4} \sum_{i=2}^{9} (p_1 + p_j - p_1 \wedge p_j)
			= 3 p_1 + \tfrac{1}{4} \sum_{i=2}^{9} (p_j - p_1 \wedge p_j) = 3 p_1
		\end{align*}
		since each of the $8$ elements $p_2,\dots,p_9$ occurs once with each sign in the sum.
	\end{proof}

	We will require idempotents associated with lines of $\cal P_3$.
	Let $L$ be any of the $12$ lines of $\cal P_3$.
	Then we define
	\begin{equation}
		e_L = - \tfrac{1}{3} \sum_{i \in L} p_i + \tfrac{1}{3} \sum_{i \not\in L} p_i \quad
		\text{and}
		\quad f_L = \id - e_L = \tfrac{2}{3} \sum_{i \in L} p_i .
	\end{equation}

	\begin{lemma}
		\label{lem pi3 line idempots}
		For each line $L$ of $\cal P_3$, $e_L$ and $f_L$ are idempotents in $A$.
		Furthermore, if $L$ and $M$ are two parallel lines in $\cal P_3$,
		then $e_L$ and $e_M$ are orthogonal, \ie $e_L e_M = 0$.
	\end{lemma}
	\begin{proof}
		Without loss of generality, we may assume that $L = \{ 1,2,3 \}$.
		We first verify that $f_L$ is idempotent:
		\begin{align*}
			f_L^2
			& = \tfrac{4}{9} (p_1 + p_2 + p_3)^2
				= \tfrac{4}{9} (p_1 + p_2 + p_3) + 2 \cdot \tfrac{4}{9} \cdot \tfrac{1}{4} \sum_{\mathclap{1 \leq i < j \leq 3}} (p_i + p_j - p_i \wedge p_j) \\
			& = \tfrac{4}{9} (p_1 + p_2 + p_3) + \tfrac{2}{9} (p_1 + p_2 + p_3)
				= \tfrac{2}{3} (p_1 + p_2 + p_3) = f_L .
		\end{align*}
		It follows that also $e_L$ is idempotent, as
		\begin{equation}
			e_Le_L = (\id-f_L)(\id-f_L) = \id - 2\id f_L + f_L = \id-f_L = e_L.
		\end{equation}

		Now notice that $e_L e_M = 0$ if and only if $f_L f_M = f_L + f_M - \id$.
		Without loss of generality, for $L$ and $M$ parallel,
		we may assume that $L = \{ 1,2,3 \}$ and $M = \{ 4,5,6 \}$.
		Then
		\begin{align*}
			f_L f_M
			& = \tfrac{4}{9} \sum_{i=1}^3 \sum_{j=4}^6 p_i p_j
				= \tfrac{1}{9} \sum_{i=1}^3 \sum_{j=4}^6 (p_i + p_j - p_i \wedge p_j) \\
			& = \tfrac{1}{9} \Bigl( 3 \sum_{i=1}^6 p_i - 3 \sum_{i=7}^9 p_i \Bigr)
				= \tfrac{1}{3} \sum_{i=1}^6 p_i - \tfrac{1}{3} \sum_{i=7}^9 p_i = f_L + f_M - \id,
		\end{align*}
        finishing the proof.
	\end{proof}

	\begin{lemma}
		\label{lem pi3 eigsp e_L}
		The eigenvalues of $e_L$ are $1,0$ and $\nicefrac{1}{2}$.
		Denote the $\al$-eigenspace of $e_L$ by $A_\alpha^{L'}$.
		Then
		\begin{align*}
			A_{1}^{L'} &= \{ \lm e_L \mid \lm \in \FF \}, \\
			A_0^{L'} &= \Bigl\{ \sum_{i \in L} \lm_i p_i + \lm \Bigl( \sum_{i \in M} p_i - \sum_{i \in N} p_i \Bigr) \mid \lm_i, \lm \in \FF \Bigr\} , \\
			A_{\nicefrac{1}{2}}^{L'} &= \Bigl\{ \Bigl( \sum_{i \in M} \lm_i p_i + \sum_{j \in N} \mu_j p_j \Bigr) \mid \textstyle \sum_i \lm_i = 0, \sum_j \mu_j = 0 \Bigr\}.
		\end{align*}
	\end{lemma}
	\begin{proof}
		To prove this, without loss of generality, we assume that $L = \{ 1,2,3 \}$, $M = \{ 4,5,6 \}$ and $N = \{ 7,8,9 \}$.
		Let $x = \sum_{i=1}^9 \lm_i p_i$ be an arbitrary element of $A$.
		Then
		\begin{multline*}
			x e_L = -\tfrac{1}{6} (\lm_4 + \dots + \lm_9)(p_1 + p_2 + p_3)
				+ \tfrac{1}{2} (\lm_4 p_4 + \dots + \lm_9 p_9) \\
				+ \tfrac{1}{6} (\lm_4 + \lm_5 + \lm_6)(p_7 + p_8 + p_9)
				+ \tfrac{1}{6} (\lm_7 + \lm_8 + \lm_9)(p_4 + p_5 + p_6) .
		\end{multline*}
		It is now straightfoward to verify that the elements occuring in the statement of the lemma
		are indeed eigenvectors for $e_L$;
		since the dimensions of these three subspaces are $4$, $4$ and $1$ respectively,
		they together span all of $A$, and hence we have found all eigenvectors.
	\end{proof}

	As a consequence of the previous result,
	we get a ``Peirce decomposition'' for $A$,
	although in fact we have not yet established whether or not $A$ is a Jordan algebra.
	\begin{corollary}
		\label{cor pi3 peirce}
		Let $\{ L_1, L_2, L_3 \}$ be a set of parallel lines in $\cal P_3$,
		and denote the corresponding idempotents by $e_1$, $e_2$ and $e_3$, respectively.
		Let $A_{ii} = A_1^{L_i} = \langle e_i \rangle$ for each $i$,
		and let $A_{ij} = A_{\nicefrac{1}{2}}^{L_i} \cap A_{\nicefrac{1}{2}}^{L_j}$ for $i \neq j$.
		Then for any choice of $\{i,j,k\} = \{1,2,3\}$, we have
		\[
			A_{ij} = \Bigl\{ \sum_{\ell \in L_k} \lm_\ell p_\ell \mid \textstyle \sum_\ell \lm_\ell = 0 \Bigr\},
		\]
		so $\dim A_{ij} = 2$, and
		\[
			A = A_{11} \oplus A_{22} \oplus A_{33} \oplus A_{12} \oplus A_{13} \oplus A_{23}. \tag*{\qed}
		\]
	\end{corollary}

	In order to describe the corresponding Jordan algebra, we will need the following definition.
	Let $E$ be the quadratic \'etale extension $E = \FF[x] / (x^2 + 3)$ of $\FF$,
	and let $\sigma \in \Gal(E/\FF)$ be its non-trivial Galois automorphism.
	This $E$ may or may not be a field,
	depending on whether $-3$ is a square in $\FF$.
	We write $E = \FF[\zeta]$ with $\zeta^2 = -3$,
	so in particular $\zeta^\sigma = -\zeta$.

	The Jordan algebra $\mathcal{H}_3(E, *)$ over $\FF$
	consists of $3\times3$ matrices over $E$ fixed by $*$,
	where $*$ is the involution on $\Mat_3(E)$
	given by conjugate transposition, \ie $(x_{ij})^* = (x_{ji}^\sigma)$;
	see Example~\ref{ex jordan}\ref{ex jordan herm}.
	(Notice that $*$ is indeed $\FF$-linear.)
	We now establish some notation for $\cal H_3(E,*)$.
	Let $e_{ij}$ be the usual matrix units in $\Mat_3(E)$.
	Following the notation in \cite{jac}*{p.\@~125}, we define
	\begin{equation}
x[ij] = x e_{ij} + x^\sigma e_{ji} \in J 
\end{equation}
	for all $i,j$; in particular, $x[ii] = (x + x^\sigma) e_{ii}$ for all $i$, and $x[ji] = x^\sigma[ij]$ for all $i,j$.
	Recall from \cite{jac}*{p.\@~126} that the multiplication in $J$
	is completely determined by the multiplication rules
	\begin{alignat}{2}
		2 x[ij] \cdot y[jk] &= xy[ik] && \text{ for all $i,j,k$ distinct}, \label{eq:M1} \\
		2 x[ii] \cdot y[ij] &= (x+x^\sigma)y[ij] && \text{ for all } i \neq j, \label{eq:M2} \\
		2 x[ij] \cdot y[ij] &= xy^\sigma[ii] + xy^\sigma[jj] && \text{ for all } i \neq j, \label{eq:M3} \\
		2 x[ii] \cdot y[ii] &= (x+x^\sigma)(y+y^\sigma)[ii] \quad && \text{ for all } i, \label{eq:M4} \\
		x[ij] \cdot y[k\ell] &= 0 && \text{ if } \{i,j\} \cap \{k,\ell\} = \emptyset .\label{eq:M5} 
	\end{alignat}
	Finally, let $J_{ij} = \{ x[ij] \mid x \in E \}$ for all $1\leq i\leq j\leq3$, so in particular
	\begin{equation}
		J = J_{11} \oplus J_{22} \oplus J_{33} \oplus J_{12} \oplus J_{13} \oplus J_{23} . 
	\end{equation}

    \medskip
	We are now fully prepared to establish the isomorphism in the following theorem.
	\begin{theorem}
		\label{thm jordan matsuo pi3}
		Assume that $\ch(\FF) \neq 2,3$.
		The Matsuo algebra $M_{\nicefrac{1}{2}}(\cal P_3)$ over $\FF$
		is isomorphic to the Jordan algebra $\cal H_3(E,*)$.
	\end{theorem}

	\begin{proof}
	    Let $A = M_{\nicefrac{1}{2}}(\cal P_3)$,
		and consider the decomposition of $A$ as in Corollary~\ref{cor pi3 peirce}.
		Let $\eta$ be the $\FF$-vector space isomorphism from $A$ to $J$
		given on each of the six Peirce subspaces by, for $\lm,\mu\in\FF$,
		\begin{align*}
			e_i &\mapsto e_{ii} = \tfrac{1}{2}[ii] \quad \text{ for all } i, \\
			\lm p_1 + \mu p_2 - (\lm + \mu) p_3 &\mapsto \bigl( \tfrac{3}{4}(\lm + \mu) + \tfrac{1}{4}(\lm - \mu) \zeta \bigr) [23] , \\
			\lm p_4 + \mu p_5 - (\lm + \mu) p_6 &\mapsto \bigl( \tfrac{3}{4}(\lm + \mu) + \tfrac{1}{4}(\mu - \lm) \zeta \bigr) [13] , \\
			\lm p_7 + \mu p_8 - (\lm + \mu) p_9 &\mapsto \bigl( \tfrac{3}{4}(\lm + \mu) + \tfrac{1}{4}(\lm - \mu) \zeta \bigr) [12].
		\end{align*}
		We will verify that $\eta$ is an isomorphism of Jordan algebras
		by going through each of the cases occuring in
		the multiplication rules \eqref{eq:M1} through \eqref{eq:M5}.

		For case~\eqref{eq:M1},
		assume that $i=1$, $j=2$ and $k=3$; the other possibilities for $i,j,k$ are similar.
		So let $x_{12} = \lm p_7 + \mu p_8 - (\lm + \mu) p_9 \in A_{12}$ and
		$y_{23} = \lm' p_1 + \mu' p_2 - (\lm' + \mu') p_3 \in A_{23}$ be arbitrary.
		Then
		\begin{align*}
				2 x_{12} y_{23}
				&= \tfrac{1}{2} \bigl( - \lm \lm' p_4 - \lm \mu' p_6 + \lm (\lm' + \mu') p_5 
					- \mu \lm' p_6 - \mu \mu' p_5 + \mu (\lm' + \mu') p_4 \\
					& \hspace*{7ex} + (\lm + \mu) \lm' p_5 + (\lm + \mu) \mu' p_4 - (\lm + \mu)(\lm' + \mu') p_6 \bigr) \\
				&= \tfrac{1}{2} \bigl( - \lm \lm' + 2 \mu \mu' + \lm \mu' + \mu \lm' \bigr) p_4
					+ \tfrac{1}{2} \bigl( 2 \lm \lm' - \mu \mu' + \lm \mu' + \mu \lm' \bigr) p_5 \\
					& \hspace*{7ex} + \tfrac{1}{2} \bigl( - \lm \lm' - \mu \mu' - 2 \lm \mu' - 2 \mu \lm' \bigr) p_6 ,
		\end{align*}
		so
		\begin{equation}
 \eta(2 x_{12} y_{23}) = 
				\bigl( \tfrac{3}{8}(\lm \lm' + \mu \mu' + 2 \lm \mu' + 2 \mu \lm') + \tfrac{3}{8}(\lm \lm' - \mu \mu') \zeta \bigr)[13] . 
\end{equation}
		On the other hand,
		\begin{align*}
				& \bigl( \tfrac{3}{4}(\lm + \mu) + \tfrac{1}{4}(\lm - \mu) \zeta \bigr)
								\cdot \bigl( \tfrac{3}{4}(\lm' + \mu') + \tfrac{1}{4}(\lm' - \mu') \zeta \bigr) \\
				&  \qquad\quad = \bigl( \tfrac{9}{16}(\lm + \mu)(\lm' + \mu') - \tfrac{3}{16}(\lm - \mu)(\lm' - \mu') \bigr) 
								+ \tfrac{3}{16} \bigl( (\lm + \mu)(\lm' - \mu') + (\lm - \mu)(\lm' + \mu') \bigr) \zeta \\
				& \qquad\quad = \bigl( \tfrac{3}{8}(\lm \lm' + \mu \mu' + 2 \lm \mu' + 2 \mu \lm') + \tfrac{3}{8}(\lm \lm' - \mu \mu') \zeta \bigr) ;
		\end{align*}
		we conclude that $\eta(2 x_{12} y_{23}) = 2 \eta(x_{12}) \eta(y_{23})$.

		The multiplication rule~\eqref{eq:M2} is equivalent to the statement that $y[ij]$ is a $\tfrac{1}{2}$-eigenvector for $e_{ii}$.
		Since $A_{ij}$ is contained in the $\tfrac{1}{2}$-eigenspace of $e_i$, it follows that $\eta(e_i y_{ij}) = \eta(e_i) \eta(y_{ij})$
		for all $i \neq j$ and all $y_{ij} \in A_{ij}$.

		We now check~\eqref{eq:M3},
		and again we assume that $i=1$ and $j=2$ since the other cases are similar.
		So let $x_{12} = \lm p_7 + \mu p_8 - (\lm + \mu) p_9 \in A_{12}$
		and $y_{12} = \lm' p_7 + \mu' p_8 - (\lm' + \mu') p_9 \in A_{23}$ be arbitrary.
		Then
		\begin{multline*}
				2 x_{12} y_{12}
				= \bigl( (\lm \lm' + \mu \mu') + \tfrac{1}{2} (\lm \mu' + \mu \lm') \bigr) (p_7 + p_8 + p_9) \\
				= \bigl( \tfrac{3}{2} (\lm \lm' + \mu \mu') + \tfrac{3}{4} (\lm \mu' + \mu \lm') \bigr) (e_1 + e_2) .
		\end{multline*}
		On the other hand,
		\begin{align*}
				& \bigl( \tfrac{3}{4}(\lm + \mu) + \tfrac{1}{4}(\lm - \mu) \zeta \bigr)
								\cdot \bigl( \tfrac{3}{4}(\lm' + \mu') + \tfrac{1}{4}(\lm' - \mu') \zeta \bigr)^\sigma \\
								& \qquad\qquad = \bigl( \tfrac{3}{4}(\lm + \mu) + \tfrac{1}{4}(\lm - \mu) \zeta \bigr)
												\cdot \bigl( \tfrac{3}{4}(\lm' + \mu') - \tfrac{1}{4}(\lm' - \mu') \zeta \bigr) \\
								& \qquad\qquad = \bigl( \tfrac{3}{4} (\lm \lm' + \mu \mu') + \tfrac{3}{8} (\lm \mu' + \mu \lm') \bigr)
										+ \tfrac{3}{8} ( \lm \mu' - \mu \lm' ) \zeta ,
		\end{align*}
		and hence
		\begin{multline*}
				\bigl( \tfrac{3}{4}(\lm + \mu) + \tfrac{1}{4}(\lm - \mu) \zeta \bigr)
								\bigl( \tfrac{3}{4}(\lm' + \mu') + \tfrac{1}{4}(\lm' - \mu') \zeta \bigr)^\sigma [ii]
				= \bigl( \tfrac{3}{2} (\lm \lm' + \mu \mu') + \tfrac{3}{4} (\lm \mu' + \mu \lm') \bigr) e_{ii} .
		\end{multline*}
		We conclude that $\eta(2 x_{12} y_{12}) = 2 \eta(x_{12}) \eta(y_{12})$.

		Case~\eqref{eq:M4} is a consequence of the definition of $x[ii] = (x + x^\sigma) e_{ii}$
		combined with the fact that $e_i$, by Lemma~\ref{lem pi3 line idempots},
		and $e_{ii}$ are idempotents.

		Finally, to deal with case~\eqref{eq:M5},
		we have to verify that $A_{ij} A_{k\ell} = 0$ as soon as $\{i,j\} \cap \{k,\ell\} = \emptyset$.
		If $i=j$ and $k=\ell$, then this again is an immediate consequence
		of Lemma~\ref{lem pi3 line idempots},
		that $e_i$ and $e_j$ are orthogonal idempotents.
		If $i=j$ and $k \neq \ell$,
		then $A_{k\ell}$ is contained in the $\tfrac{1}{2}$-eigenspace of both $e_k$ and $e_\ell$,
		and hence in the $0$-eigenspace of $\id - e_k - e_\ell = e_i$;
		it follows that $A_{ii} A_{k\ell} = 0$.
	\end{proof}

    \begin{remark}
    \begin{compactenum}
        \item
            If we knew in advance that $M_{\nicefrac{1}{2}}(\cal P_3)$ was a Jordan algebra,
            the calculations in the proof of Theorem~\ref{thm jordan matsuo pi3}
            could be replaced by an application of Jacobson's ``Strong Coordinatization Theorem'',
            \cite{jac}*{Theorem 5, p.\@~133}.
            Indeed, the idempotents $e_1$, $e_2$ and $e_3$ are {\em strongly connected},
            and the {\em coordinatizing algebra},
            an algebra structure on $A_{ij}$ for $i\neq j$,
            is isomorphic to the $\FF$-algebra $E$,
            which can be obtained from \cite{jac}*{Lemma 3, p.\@~135}.
            This is how we obtained the formulas for the isomorphism $\eta$.
        \item
            Jeroen Demeyer pointed out to us that there is another way to describe the extension $E$, namely as $E = \FF[x]/(x^2 + x + 1)$.
            Let $\beta$ be one of the roots of $x^2 + x + 1$ in $E$, and let $\xi = \beta+1$.
            Notice that $\beta$ has order $3$ since $\ch(\FF) \neq 3$, and $\xi^2 = \beta$ so $\xi$ has order $6$; we can think of $\xi$ as a $6$-th root of unity
            (even though $E$ might not be a field).
            Then the formulas for the isomorphism $\eta$ occuring in the beginning of the proof of Theorem~\ref{thm jordan matsuo pi3} take the form
			\[ \lm p_1 + \mu p_2 + \nu p_3 \mapsto \bigl( \lm \xi + \mu \xi^5 + \nu \xi^3 \bigr) [23] = \bigl( \lm \xi + \mu \xi^\sigma - \nu \bigr) [23] , \]
            for all $\lm,\mu,\nu \in \FF$ with $\lm + \mu + \nu = 0$ (and similarly for the other expressions).
    \end{compactenum}
    \end{remark}

\subsection{The case $\ch(\FF) = 3$}\label{ss:eq3}

	In order to describe the situation in the case of $\ch(\FF)=3$,
	we require some more definitions from the theory of Jordan algebras.
	\begin{definition}[\cite{jac}]
		Let $J$ be a Jordan algebra over $\FF$.
		\begin{compactenum}
			\item
				For every $a,b \in J$, we define $U_a(b) = 2a(ab) - a^2 b$.
				This defines, for each $a$, a linear map $U_a \colon J \to J$,
				known as the {\em $U$-operator} of $a$.
			\item
				An element $a \in J$ is called an {\em absolute zero divisor} if $U_a$ is the zero map.
			\item
				An element $a \in J$ is called {\em trivial} if $U_a$ is the zero map and moreover $a^2 = 0$.
		\end{compactenum}
	\end{definition}

	\begin{theorem}
		\label{thm jordan matsuo pi3-ch3}
		The Matsuo algebra $M_{\nicefrac{1}{2}}(\cal P_3)$ over a field $\FF$ of characteristic $3$
		is isomorphic to a $9$-dimensional non-unital Jordan algebra with an $8$-dimensional radical $R$.
		Furthermore, there is a chain of ideals of $J$
		\begin{equation}
			0 < Z < T < R < J \quad \text{ with }
			\dim Z = 1, \; \dim T = 6, \text{ and }
			Z = T^2, \; T = R^2,
		\end{equation}
		such that the elements of $Z$ are trivial,
		the elements of $T$ are absolute zero divisors,
		and $J/R$ is a unital Jordan algebra isomorphic to $\FF$.
	\end{theorem}
	\begin{proof}
		Let $A = M_{\nicefrac{1}{2}}(\cal P_3)$.
		As previously, we will use the affine plane $\cal P_3$ from Figure~\ref{fig-projpl},
		and for each $i \in \{ 1,\dots, 9 \}$, we let $p_i \in A$ be the generator
		corresponding to the point $i$ in Figure~\ref{fig-projpl}.
		Since $\ch(\FF) = 3$, however,
		the element $z = \sum_{i=1}^9 p_i$ is an annihilating element of the algebra $A$,
		\ie $zx = 0$ for all $x \in A$.
		In particular, $A$ is non-unital.

		It is a straightforward but lengthy calculation to verify
		that the linearised Jordan identity, see \eg \cite{mccrimmon}, Proposition 1.8.5 (1),
		\begin{equation}
			((xz)y)w + ((zw)y)x + ((wx)y)z
				= (xz)(yw) + (zw)(yx) + (wx)(yz),
		\end{equation}
		holds over $\FF_3$ and hence $A$ is a Jordan algebra over $\FF_3$
		and over any field extension of $\FF_3$,
		that is, over any field of characteristic $3$.
		We performed this check by computer.


		It is now easy to verify that
		\begin{equation}
			Z = \la p_1+\dotsm+p_9 \ra, \quad
			T = \la p_i+p_j+p_k\mid\{i,j,k\}\text{ a line}\ra, \quad
			R = \Bigl\la \sum_{i=1}^9\lm_i p_i\mid\sum_{i=1}^9\lm_i=0 \Bigr\ra, 
		\end{equation}
		are a chain of ideals in $A$,
		with $\dim Z = 1, \dim T = 6, \dim R = 8$,
		and that
		\begin{equation}
			R^2 = T,\quad
			T^2 = Z,\quad
			Z^2 = 0,
		\end{equation}
		and therefore $R$ is a solvable ideal.
		We already showed that $z$, spanning $Z$, is trivial.
		To show that $T$ consists of absolute zero divisors,
		let $t\in T$ be arbitrary.
		As $T^2 = Z$ we have $tt\in Z$ and hence $(tt)x = 0$ for all $x\in A$.
		To show that $t(tx) = 0$, by linearity of $x\in A$ we may take $x = p_i$ a point.
		If $\ell$ is a line containing $p_i$, then $\Bigl(\sum_{p\in\ell}p\Bigr)p_i = 0$;
		if $p_i\not\in\ell$, then $p_i$ lies in one of the two lines $\ell',\ell''$ parallel to $\ell$,
		say $\ell'$, and $\Bigl(\sum_{p\in\ell}p\Bigr)p_i = \sum_{p\in\ell}p - \sum_{p\in\ell''}p$.
		Let $m$ be any line; then
		\begin{equation}
			\Bigl(\sum_{q\in m}q\Bigr) \Bigl(\sum_{p\in\ell}p - \sum_{p\in\ell''}p\Bigr) = 0.
		\end{equation}
		As $t$ is a linear sum of terms of the form $\sum_{p\in\ell}p$ for $\ell$ a line,
		this shows that $t(tp_i) = 0$ for all $i$
		and hence $t(tx) = 0$ for all $x\in A$.


		The quotient $J/R$ is a $1$-dimensional algebra spanned by the image $\bar p_1$ of $p_1 \in J$,
		which satisfies $\bar p_1 \cdot \bar p_1 = \bar p_1$,
		so $J/R$ has unit $\bar p_1$
		and is isomorphic to the Jordan algebra of the field $\FF$.
		Thus $J$ has a nonsolvable quotient,
		and in particular is not itself solvable.
		As $R$ has codimension~$1$,
		$R$~is the maximal solvable ideal, \ie the radical, of $J$.
			%
			%
	\end{proof}


\section{Classification of Jordan Matsuo algebras}
	\label{sec-class}

    In this final section, we will show that there are essentially no other examples of Jordan Matsuo algebras than the ones we have presented;
    see Theorem~\ref{thm jordan matsuo} below.

	In the proof, we will require two additional (well-known) definitions.
    \begin{definition}
        Let $G$ be a group, and let $D \subseteq G$ be a subset.
	    The {\em noncommuting graph} on $D$
    	is the graph $\cal D$ with vertices $D$
    	and edges $\{c,d\}$ for $c,d\in D$ with $[c,d]\neq 1$.
    \end{definition}
    If $(G,D)$ is a $3$-transposition group, then the noncommuting graph $\cal D$ on $D$
    is related to the Fischer space $\Fspace$ of $(G,D)$:
	it is precisely its {\em collinearity graph},
    i.e.\@ the graph with vertex set equal to the point set of $\Fspace$, where two distinct vertices are connected by an edge
    if and only if they are collinear in the geometry $\Fspace$.

	Such noncommuting graphs can also be seen as (simply-laced) Coxeter diagrams.
    The following definition provides some sort of converse in this respect.

    \begin{definition}
    	Let $\cal D$ be a graph in which two vertices lie on at most a single edge.
    	Let $D'$ be a set of generators in bijection with the vertices of $\cal D$,
    	and let $G$ be the group with presentation
    	\begin{equation}
    		G = \bigl\langle d\in D' \mid d^2 = 1, \quad
    		\size{cd} = 3 \text{ for }c\in D'\text{ collinear, }
    		\size{cd}=2 \text{ otherwise}
    		\bigr\rangle,
    	\end{equation}
    	and set $D = {D'}^G$.
    	We call the pair $(G,D)$ the {\em Coxeter group $\op{Cox}(\cal D)$ on $\cal D$}.
    \end{definition}

    We now come to our classification result.
	\begin{theorem}
		\label{thm jordan matsuo}
		Let $J$ be a finite-dimensional Jordan algebra over $\FF$
		which is also a Matsuo algebra $M_{\nicefrac{1}{2}}(\Fspace)$.
		Then $\Fspace$ is a disjoint union of Fischer spaces $\cal P_3$ and $\Fspace(\rt A_n)$, $n \geq 1$.
		In particular, $J$ is a direct product of the Jordan algebras described in
		Theorems~\ref{thm jordan matsuo an}, \ref{thm jordan matsuo pi3} and \ref{thm jordan matsuo pi3-ch3}.
	\end{theorem}
	\begin{proof}
		Suppose that $\Fspace$ is a Fischer space
		and that $\Fspace = \Fspace_0\cup\Fspace_1$
		is a partition into two mutually disconnected nontrivial Fischer spaces.
		Then $M_\al(\Fspace) = M_\al(\Fspace_0) \times M_\al(\Fspace_1)$ as $\FF$-algebras.
		Hence we may assume without loss of generality that $\Fspace$ is connected,
		and we proceed to show that $\Fspace$ is either $\Fspace(\rt A_n)$ for some $n \geq 1$ or $\cal P_3$.

		Recall that the rank of a connected Fischer space $\Fspace$
		is the size of a smallest collection of points generating $\Fspace$.
		We proceed case-by-case for Fischer spaces of rank at most $4$.
		A connected Fischer space of rank $1$ is a single point,
		and the associated $1$-dimensional algebra is obviously Jordan;
		this is the case $\Fspace = \Fspace(\rt A_1)$.
		A connected Fischer space of rank $2$ is a line, generated by two points,
		so the only Matsuo-Jordan algebra here is $M_{\nicefrac{1}{2}}(\Fspace(\rt A_2))$.
		Positive answers in rank $3$,
		that is, for $\cal P_2^\vee\cong \Fspace(\rt A_3)$ and $\cal P_3$,
		are given by Theorems~\ref{thm jordan matsuo an}, \ref{thm jordan matsuo pi3} and \ref{thm jordan matsuo pi3-ch3}.
		Recall that these are the only Fischer spaces of rank 3 by definition.

		The rank $4$ Fischer spaces are classified by \cite{hall-genth3}*{Proposition 2.9}.
		They are the Fischer space $\Fspace(\rt A_4)$
		and the Fischer spaces of quotients of the $3$-transposition groups
		\begin{equation}
			\label{eq 4rk fs}
			\begin{gathered}
				(W_2(\affrt A_3),D_2), \quad
				(W_3(\affrt A_3),D_3), \\
				(G_4 = 2^{1+6}:SU_3(2)',D_4), \quad
				\text{ and M. Hall's } (G_5 = 3^{10}:2,D_5).
			\end{gathered}
		\end{equation}
		To define $W_k(\affrt A_n)$,
		let $G'$ be the $\FF_k$-linear permutation representation of $\Sym(n+1)$,
		that is, the semidirect product of $\Sym(n+1)$
		with the module $\FF_k^{n+1}$,
		where the action is by permutation of the standard ordered basis $\{v_1,\dotsc,v_{n+1}\}$ of $\FF_k^{n+1}$,
		and let $D'$ be the conjugacy class of the image of $(1,2)\in\Sym(n+1)$
		in the semidirect product $G'$.
		Then $W_k(\affrt A_n)$ is the quotient $(G,D)$
		of $(G',D')$ by the diagonal $\la v_1+\dotsm+v_{n+1}\ra$.

		The latter two groups are defined by presentations.
		Let $\cal C$ be the complete graph on $\{a,b,c,d\}$,
		and $\cal C'$ the graph obtained from $C$ by deleting the edge $\{b,c\}$.
		Then
		\begin{align}
			G_4 & = \op{Cox}(\cal C') / \bigl((a^bd)^3 = (a^cd)^3 = (a^{bc}d)^3 = 1 \bigr), \\
			G_5 & = \op{Cox}(\cal C) / \bigl((b^cd)^3 = (a^bc)^3 = (a^bd)^3 = (a^cd)^3
				 = (a^{bd}c)^3 = (a^{cd}b)^3 = (a^{dc}b)^3 = 1\bigr).
		\end{align}
		Let $D_i$ be the image of the Coxeter involutions closed under conjugation
		in the above quotient for $i=4,5$.
		Then $D_i$ generates $G_i$ and $(G_i,D_i)$ is a $3$-transposition group.

		It follows by Theorem~\ref{thm jordan matsuo an} that the Matsuo algebra
		of $\Fspace(\rt A_4)$ is Jordan.
		This is the only one out of the five groups
		which gives a Jordan algebra.
		For the others, we will, in each case, choose $4$ generating transpositions $a,b,c,d$ of $G$,
		where $G$ is one of $W_2(\affrt A_3),W_3(\affrt A_3),G_4$ or $G_5$,
		such that for $x =	a +	b +	c$ in the algebra $M_{\nicefrac{1}{2}}(\Fspace)$,
		we find $(xx)( dx) \neq ((xx) d)x$,
		whence $A$ is not Jordan.
		We show that $A$ is not Jordan for $G = W_k(\affrt A_3)$ and $k = 2,3$
		by explicit calculation:
		set $a',b',c',d'$ equal to
		\begin{equation}
			\begin{pmatrix}[0.8]
				0 & 1 & 0 & 0 & 0 \\
				1 & 0 & 0 & 0 & 0 \\
				0 & 0 & 1 & 0 & 0 \\
				0 & 0 & 0 & 1 & 0 \\
				0 & 0 & 0 & 0 & 1
			\end{pmatrix},\quad
			\begin{pmatrix}[0.8]
				1 & 0 & 0 & 0 & 0 \\
				0 & 0 & 1 & 0 & 0 \\
				0 & 1 & 0 & 0 & 0 \\
				0 & 0 & 0 & 1 & 0 \\
				0 & 0 & 0 & 0 & 1
			\end{pmatrix},\quad
			\begin{pmatrix}[0.8]
				1 & 0 & 0 & 0 & 0 \\
				0 & 1 & 0 & 0 & 0 \\
				0 & 0 & 0 & 1 & 0 \\
				0 & 0 & 1 & 0 & 0 \\
				0 & 0 & 0 & 0 & 1
			\end{pmatrix},\quad
			\begin{pmatrix}[0.8]
				0 & 0 & 0 & 1 & 0 \\
				0 & 1 & 0 & 0 & 0 \\
				0 & 0 & 1 & 0 & 0 \\
				1 & 0 & 0 & 0 & 0 \\
				1 & 0 & 0 & -1 & 1
			\end{pmatrix},
		\end{equation}
		respectively;
		then $\{ a', b', c', d'\}$ generates $G' = k^4:\Sym(4)$,
		and if $n = 
			\left(
			\begin{array}{c|c}
				I_4 & 0 \\
				\hline
				1 & 1 \\
			\end{array}
			\right)$,
		then $G = G'/\la n\ra$,
		$a = a'/\la n\ra$ and likewise define $b,c,d$.
		It is easy to check that $a,b,c,d$ are conjugate,
		and act as transpositions on $k^3\subseteq W_k(\affrt A_n)$,
		so that $a,b,c,d\in D$.
		For $(G,D) = W_2(\affrt A_3)$,
		the coefficient of $ a$ in $((xx) d)x$ is $\frac{3}{8}$
		and the coefficient of $ a$ in $(xx)( dx)$ is $\frac{7}{16}$.
		For $(G,D) = W_3(\affrt A_3)$, the respective coefficients are $\frac{13}{32}$ and $\frac{7}{16}$.
		We see that $\frac{3}{8},\frac{13}{32}\neq\frac{7}{16}$ in any characteristic (not $2$ by assumption).
		In each case this shows that the Jordan identity does not hold.

		Abusively, let now $a,b,c,d$ stand for the images of $a,b,c,d$
		under the quotient $\op{Cox}(\cal C')\to G_4$
		or $\op{Cox}(\cal C)\to G_5$.
		Then $x =	a +	b +	c$ in the algebra
		again gives $((xx) d)x \neq (xx)( d x)$.
		In both cases, the idempotent corresponding to $a^{cdb}$
		has a nonzero contribution, namely with coefficient $-\frac{1}{32}$,
		on only the lefthand side.
		Therefore the Matsuo algebras of the Fischer spaces of $(G_4,D_4)$ and $(G_5,D_5)$
		are also not Jordan algebras.

		This handles the cases when $(G,D)$ is one of the $3$-transposition groups in \eqref{eq 4rk fs}.
		These groups admit finitely many quotients with rank $4$ Fischer spaces,
		for which the same method shows that
		the coefficients of some element in $(xx)(dx)$ and $((xx)d)x)$
		differ by $\frac{1}{32}$ or $\frac{1}{64}$,
		whence these quotients cannot give rise to Jordan algebras either.

		Suppose that $(G,D)$ is a transposition group
		whose Fischer space $\Fspace$ has rank $r$ at least $5$,
		such that the Matsuo algebra $A = M_{\nicefrac{1}{2}}(G,D)$
		is Jordan.
		If $T\subseteq D$ is a generating set for $G$
		and $\cal T$ is the noncommuting graph on $T$,
		then $G$ is a quotient of the Coxeter group on $\cal T$.
		Suppose that the subspace spanned by $T' = \{d_1,\dotsc,d_4\}$
		has rank $4$ in $\Fspace$.
		By the above, $\la T'\ra\cong\Sym(5)$
		and $\cal T'$ is a line with $4$ nodes,
		since the subalgebra of $A$
		generated by $ d_1,\dotsc, d_4$ must itself be Jordan.
		Therefore if $T = \{d_1,\dotsc,d_r\}\subseteq D$
		is a set of generators for $G$
		(which is connected, since the noncommuting graph on $D$ is connected),
		then no vertex has valency $3$ in $\cal T$.
		Therefore $\cal T$ is either a line or a loop,
		corresponding to $\rt A_r$ or or $\affrt A_{r-1}$.
		By Theorem~\ref{thm jordan matsuo an},
		$M_{\nicefrac{1}{2}}(\Fspace(\rt A_r))$ is Jordan.
		Suppose $\cal T$ is $\affrt A_{r-1}$.
		Then $G$ is a quotient of $W_k(\affrt A_{r-1})$ \cite{hall-genth3}*{p.\@ 272}.
		But $W_k(\affrt A_{r-1})$ admits an embedding of $W_k(\affrt A_3)$
		for all $r\geq 5$: for
		\begin{equation}
			\begin{gathered}
				a' = \begin{pmatrix}[0.8]
					0 & 1 \\
					1 & 0 \\
				\end{pmatrix}\oplus I_{r-2},\quad
				b' = (1)\oplus
				\begin{pmatrix}[0.8]
					0 & 1 \\
					1 & 0 \\
				\end{pmatrix}\oplus I_{r-3},\quad
				c' = I_2\oplus
				\begin{pmatrix}[0.8]
					0 & 1 \\
					1 & 0 \\
				\end{pmatrix}\oplus I_{r-4},\quad \\
				d' = \left(
				\begin{array}{cccc|c|c}
					0 & 0 & 0 & 1 & 0 & 0 \\
					0 & 1 & 0 & 0 & 0 & 0 \\
					0 & 0 & 1 & 0 & 0 & 0 \\
					1 & 0 & 0 & 0 & 0 & 0 \\
					\hline
					0 & 0 & 0 & 0 & I_{r-5} & 0 \\
					\hline
					1 & 0 & 0 & -1 & 0 & 1 \\
				\end{array}
				\right),\quad
				n = \left(
				\begin{array}{c|c}
					I_{r-1} & 0 \\
					\hline
					1 & 1 \\
				\end{array}
				\right),
			\end{gathered}
		\end{equation}
		we have that $W_k(\affrt A_{r-1})$
		is the quotient of $k^r:\Sym(r)$ by $\la n\ra$,
		and $a,b,c,d$ the images of $a',b',c',d'$ in $W_k(\affrt A_{r-1})$
		generate a subgroup $W_k(\affrt A_3)$.
		Therefore the Matsuo algebra of $W_k(\affrt A_3)$ is a subalgebra of $A$,
		which is not Jordan, so $A$ is not Jordan.
		Hence the only possibility in rank $r \geq 5$ is that $\cal T$ is $\rt A_r$.
	\end{proof}





\end{document}